\documentclass{amsart}

\makeatletter
 \def\LaTeX{\leavevmode L\raise.42ex
   \hbox{\kern-.3em\size{\sf@size}{0pt}\selectfont A}\kern-.15em\TeX}
\makeatother

\newcommand{\BibTeX}{{\rm B\kern-.05em{\sc
i\kern-.025emb}\kern-.08em\TeX}}
\newtheorem{col}{Corollary}[section]
\newtheorem{thm}{Theorem}[section]
\newtheorem{defn}{Definition}

\newtheorem{rem}[thm]{Remark}
\theoremstyle{defn}
\newtheorem{lem}[thm]{Lemma}

\newtheorem{assump}[thm]{Assumption}
\numberwithin{equation}{section}

\begin{document}

\title[Sampling by average values and average  splines  on Dirichlet spaces]{Sampling by averages and average  splines   on Dirichlet  spaces and  on combinatorial graphs}

\author{Isaac Z. Pesenson}
\address{Department of Mathematics, Temple University,
Philadelphia, PA 19122} \email{pesenson@temple.edu}

\maketitle

\begin{abstract}
In the framework of a  strictly local regular Dirichlet space ${\bf X}$ we introduce the subspaces   $PW_{\omega},\>\>\omega>0,$  of  Paley-Wiener  functions of bandwidth $\omega$. It is shown that every function in 
$PW_{\omega},\>\>\omega>0,$ is uniquely determined by its average values over a family of balls $B(x_{j}, \rho),\>x_{j}\in {\bf X},$  which form  an admissible   cover of ${\bf X}$ and whose radii are comparable to $\omega^{-1/2}$. The entire development heavily depends on some local and global Poincar\'e-type  inequalities. 
In the second part of the paper we realize the same idea in the setting of  a weighted combinatorial finite or infinite countable graph  $G$.
We have to treat the case of graphs separately  since  the Poincar\'e inequalities we are using on them are somewhat different from the Poincar\'e inequalities in the first part. 

\end{abstract}
\section{Introduction}

We consider a metric measure space ${\bf X}$ with doubling property and a self-adjoint operator $\mathcal{L}$ in $L_{2}({\bf X})$ which governs local geometry of ${\bf X}$ through a Poincare-type inequality. In fact, we are working in the environment  of the so-called strictly local regular Dirichlet spaces \cite{   A}, \cite{BH},  \cite{      FOT}, \cite{S1}-\cite{ S3}.  Following \cite{Pes88}  -\cite{Pes01} we introduce the subspaces     $PW_{\omega}(\mathcal{ L}),\>\>\omega>0,$ of  Paley-Wiener  functions in $L_{2}({\bf X})$ and then  show that every function $f$  in  $PW_{\omega}(\mathcal{ L}),\>\>\omega>0,$ is uniquely determined by its average values over a family of balls $B(x_{j}, \rho),\>x_{j}\in {\bf X},$  which form  an admissible   cover of ${\bf X}$ and whose radii are comparable to $\omega^{-1/2}$. Reconstruction methods for reconstruction of an $f\in PW_{\omega}(\mathcal{L})$ from appropriate set of its averages are introduced. One method is using language of Hilbert frames. Another one is using average variational interpolating splines which are constructed in the setting of  metric measure spaces.  It is shown that every Paley-Wiener function  is a limit of specific variational splines which have the same averages over balls $B(x_{j},\rho)$ as $f$.

In the second part of the paper we realize the same idea in the setting of  a weighted combinatorial finite or infinite countable graph  $G$.  Although the second part is very close to the first one ideologically, it is somewhat different technically  and can be read independently of the first part. 

In both parts of the paper we strongly rely on the local Poincar\'e (\ref{Poinc-0}) and a global Poincar\'e  (\ref{GPin}) and a local Poincar\'e-type (\ref{Poinc-g}) and a global Poincar\'e-type (\ref{main-ineq}) inequalities.  It is worth to notice that these inequalities in the case of graphs are not exactly the same as in the case of Dirichlet spaces. It explains why we treat graphs separately. A detailed comparison of Poincar\'e inequalities we are using on Dirichlet spaces and on graphs  is given in Remark \ref{exp}.

It should me mentioned that the idea to use certain local information (other than values at sampling point) for reconstruction of bandlimited functions on graphs was already explored in   \cite{W}. However, the results and methods of \cite{W} and of our paper are very different.  We also want to mention that methods of the present paper are similar to methods of our paper \cite{Pes04b} in which sampling by average values and average splines were developed on Riemannian manifolds.

\section{Metric measure spaces and Poincare inequality}\label{assumption}

In this section we describe a class of metric measure spaces  we are going to work with. Actually, we assume many properties but the most principle of them are: the doubling property (\ref{doubling}), existence of a self-adjoint non-negative operator in the corresponding $L_{2}({\bf X})$, and  the Poincar\'e inequality (\ref{Poinc-0}). 

{\bf Main assumptions about metric measure space}. 

A homogeneous space in the sense of Coifman and Weiss \cite {CW1}, \cite{CW2} is a triple $(\mathbf{X}, dist, \mu)$  where $\mathbf{X}$ is a set which is equipped with a metric $dist$ and a positive Radon measure $\mu$ such that the so-called  doubling condition holds. Namely, there exists a $D>0$ such that for every open ball $B(x,r)$ with center $x\in \mathbf{X}$ and radius $r>0$ 
\begin{equation}\label{doubling}
0<\left|B(x, 2r)\right|\leq 2^{D}\left|B(x, r)\right|,\>\>\> \quad \mbox{where} \quad \left|B(x, s)\right|=\mu\left(B(x, s)\right).
\end{equation}
It is very common to assume that $dist$  defines a locally compact separable topology on ${\bf X}$. 

In the first part of our paper we will work with the so-called \textit{strictly local regular Dirichlet spaces.} We just outline the main points of this framework and refer to \cite{   A}, \cite{BH},  \cite{      FOT}, \cite{S1}-\cite{ S3} where all the   missing  details can be found.  

The major assumption is that  the real space $L_{2}(\mathbf{X}, \mu)$ is equipped with a non-negative self-adjoint operator $\mathcal{L}$ whose domain $\mathcal{D}(\mathcal{L})$ is dense in $L_{2}(\mathbf{X}, \mu)$.  We consider associated  symmetric bilinear  form $\mathcal{E}$ with domain $\mathcal{D}(\mathcal{E})\supset\mathcal{D}({\mathcal{L}})$ which is defined by the formula
$$
\mathcal{E}(f, g)=\left<\mathcal{L}f, g\right>=\left<f, \mathcal{L}g\right>=\mathcal{E}(g, f),\>\>\>f,g\in \mathcal{D}({\mathcal{E}}).
$$
Assuming that $\mathcal{E}$ is a strongly local regular Dirichlet form one can show existence of a signed measure $d\Gamma$ which is a bilinear map defined on $\mathcal{D}(\mathcal{E})\times\mathcal{D}(\mathcal{E})$ such that 
$$
\mathcal{E}(f,g)=\int_{{\bf X}}d\Gamma(f,g),\>\>\>f,g\in \mathcal{D}(\mathcal{E}).
$$
In particular, for the non-negative square root $\mathcal{L}^{1/2}$ one has
\begin{equation}\label{1/2}
\|\mathcal{L}^{1/2}f\|^{2}=\mathcal{E}(f,f)=\int_{{\bf X}}d\Gamma(f,f),\>\>\>f\in \mathcal{D}(\mathcal{E}).
\end{equation}

\begin{rem}
In one of the most "real life" situations when ${\bf X}$ is a Riemannian manifold and $\mathcal{L}$ is the corresponding Laplace-Beltrami operator the measure $d\Gamma(f,g)$ is given by a formula
\begin{equation}\label{Gamma}
d\Gamma(f,g)(x)=\Gamma(f,g)(x)d\mu,
\end{equation}
where the function $\Gamma(f,g)$ for smooth compactly supported $f$ and $ g$ is defined by the formula
\begin{equation}\label{nabla}
\Gamma(f,g)=\frac{1}{2}\left(\mathcal{L}(fg)-f\mathcal{L}g-g\mathcal{L}f\right).
\end{equation}
In fact, the last formula holds even in general case if one imposes some extra conditions on the domain of $\mathcal{L}$.
\end{rem}

Our next important assumption is about {\bf a local Poincar\'e inequality}.

\begin{assump} \label{ass1}There exists  constant $C>0$ such that for any $f\in \mathcal{D}(\mathcal{L})$ and any  ball $B(x, r),\>\>x\in \mathbf{X},$ of a sufficiently small radius $r$ the following Poincare inequality holds
\begin{equation}\label{Poinc-0}
\int_{B(x,r)}|f(y)-f_{B}|^{2}d\mu(y)\leq Cr^{2}\int_{B(x,r)}\Gamma(f,f)d\mu,
\end{equation}
where 
$$f
_{B}=\frac{1}{|B(x,r)|}\int_{B(x,r)}f(x)d\mu(x).
$$
\end{assump}

The above assumptions constitute our framework for the next section. In section \ref{splines} we will add another assumption to this list.

\section{A global Poincar\'e inequality }

\subsection{A covering lemma}

\begin{lem}\label{covering}
There exists a constant $N=N(\mathbf{X})$ and for every sufficiently small $\rho>0$ there exists a set of points $\mathbf{X}_{\rho}=\{x_{j}\},\>x_{j}\in {\bf \mathbf{X}}$, such that

	\begin{enumerate}

	\item  balls $B(x_{j}, \rho/2)$ are disjoint 
	$
	B(x_{j},\rho/2)\cap B(x_{i}, \rho/2)=\emptyset,\>\>\>\>\>j\neq i,
	$ 
	\item  multiplicity of the cover $\left\{B(x_{j}, \rho)\right\}$ is not greater than $\leq N(\mathbf{X})$.

\end{enumerate}
\end{lem}

\begin{proof}

Let us choose a family of disjoint balls $B(x_{i},\rho/2)$ such
that there is no ball $B(x,\rho/2), x\in M,$ which has empty
intersections with all balls from our family. Then the family
$B(x_{i},\rho)$ is a cover of $M$. Every ball from the family
$\{B(x_{i}, \rho)\}$, that has non-empty intersection with a
particular ball $B(x_{j}, \rho)$ is contained in the ball
$B(x_{j}, 2\rho)$. Since any two balls from the family
$\{B(x_{i},\rho/2)\}$ are disjoint, it gives the following
estimate for the index of multiplicity $N$ of the cover
$\{B(x_{i},\rho)\}$:
\begin{equation}\label{N}
N\leq \sup_{y\in \mathbf{X}}\frac{|B(y, 2\rho)|}{\inf_{x\in B(y, 2\rho)|}|B(x, \rho/2)|}.
\end{equation}
Let us notice that for any $m\in \mathbb{N}$ the doubling property (\ref{doubling}) implies 
$$
\left|B(x,  2^{m}r)\right|\leq 2^{mD}|B(x, r)|,\>\>\>x\in \mathbf{X}, \>\>r>0, \>\>\rho>0, 
$$
and if for a $\rho>0$ one has $2^{m}\leq \rho<2^{m+1}$ then
\begin{equation}\label{1a}
\left|B(x,  \rho r)\right|< |B(x, 2^{m+1}r)|\leq 2^{(m+)D}|B(x, r)|\leq 2^{D}\rho^{D}|B(x, r)|.
\end{equation}
Now, the last inequality along with the  obvious  inclusion $B(x,r)\subset B(y, r+dist(x,y))$ implies
\begin{equation}\label{1b}
\left|B(x, r)\right|\leq 2^{D}\left(1+dist(x,y)/r\right)^{D}|B(y, r)|,\>\>\>x,y\in \mathbf{X}, \>\>r>0.
\end{equation}
Let's return to  the inequality (\ref{N}).         Since $x\in B(y, 2\rho)$ we have $dist(x,y)\leq 2\rho$ and according to (\ref{1b}) 
$$
|B(y, \rho/2)|\leq 2^{D}\left( 1+\frac{dist(x,y)}{\rho/2}\right)^{D} |B(x, \rho/2)|\leq 10^{D} |B(x, \rho/2)|.
$$
In other words, for any $y\in \mathbf{X}$ and every $x\in B(y, 2\rho)$  one has
\begin{equation}\label{reverse}
|B(x, \rho/2)|\geq 10^{-D}|B(y, \rho/2)|.
\end{equation}
This inequality along with (\ref{1a}) allows to continue estimation
of $N{\bf X})$:

$$N({\bf X})\leq  10^{D} 
\sup_{y\in \mathbf{X}}\frac{|B(y,2\rho)|}{|B(x,\rho/2)|} \leq
10^{D}\sup_{y\in \mathbf{X}}\frac{8^{D}|B(y,\rho/2)|}{|B(y, \rho/2)|}=80^{D}.
$$
Lemma is proved.

\end{proof}

\begin{defn}\label{def-cov}
Every set   $X_{\rho}=\{x_{j}\},\>x_{j}\in {\bf \mathbf{X}},\>\rho>0,$ which satisfies properties of the previous lemma will be called a {\bf metric $\rho$-lattice}.
\end{defn}

\subsection{A global Poincar\'e inequality and its implications}

We will need the  inequality (\ref{in}) below. 
One has  for all $\alpha>0$
$$
|A|^{2}\leq \left(|A-B|+|B|\right)^{2}\leq |A-B|^{2}+2|A-B||B|+|B|^{2},
$$
and
$$
2|A-B||B|\leq\alpha^{-1}|A-B|^{2}+\alpha|B|^{2},\>\>\>\>\alpha>0,
$$
which imply  the inequality 
\begin{equation}\label{in}
(1+\alpha)^{-1}|A|^{2}\leq\alpha^{-1}|A-B|^{2}+|B|^{2},\>\>\alpha>0.
\end{equation}
Let $X_{\rho}=\{x_{j}\}$ be a $\rho $-lattice and 
$$
 \zeta_{j}(x)=|B(x_{j}, \rho)|^{-1}\chi_{j}(x),
 $$ 
 where $\chi_{j}$ is characteristic function of $B(x_{j}, \rho)$.
We have
$$
|f|^{2}\leq (1+\alpha)\left|f-\left<f, \zeta_{j}\right>\right|^{2}+\frac{1+\alpha}{\alpha}\left|\left<f, \zeta_{j}\right>\right|^{2},\>\>\>\alpha>0,
$$
and then
$$
\|f\|^{2}\leq \sum_{j}\int_{B(x_{j}, \rho)}|f|^{2}\leq (1+\alpha)\sum_{j}\int_{B(x_{j}, \rho)}|f-     \left<f, \zeta_{j}\right>|^{2}+      \frac{1+\alpha}{\alpha}  \sum_{j}| \left<f, \xi_{j}\right>|^{2},
$$
where 
\begin{equation}\label{xi}
\xi_{j}=|B(x_{j}, \rho)|^{-1/2}\chi_{j}.
\end{equation}
Since by the very definition of a $\rho$-lattice the corresponding cover by balls $B(x_{j}, \rho)$ has finite multiplicity $\leq N(\mathbf{X})$ we obtain according to (\ref{Poinc-0}) and (\ref{1/2})
$$
\sum_{j}\int_{B(x_{j},\rho) }|f-          \left<f, \zeta_{j}\right>    |^{2}\leq C\rho^{2}\sum_{j}\int_{B(x_{j},\rho) }d\Gamma(f,f)\leq 
$$
$$
CN\rho^{2}\int_{\bf X }\left|\mathcal{L}^{1/2}f\right|^{2}\leq  CN(\mathbf{X})\rho^{2}\|\mathcal{L}^{1/2}f\|^{2}.
$$
Thus we can formulate the following result.
\begin{thm}\label{GPI}For every $\rho$-lattice $\{x_{j}\}$ and the corresponding set of functions $\{\xi_{j}\}$ defined in (\ref{xi}) the following inequality holds for any $f\in \mathcal{D}(\mathcal{L}^{1/2})$
\begin{equation}\label{GPin}
\|f\|^{2}\leq (1+\alpha)CN(\mathbf{X})\rho^{2}\|\mathcal{L}^{1/2}f\|^{2}+  \frac{1+\alpha}{\alpha}  \sum_{j}| \left<f, \xi_{j}\right>|^{2},\>\>\>\alpha>0,
\end{equation}
where $C$ is the same as in (\ref{Poinc-0}).
We call this inequality  {\bf a global Poincar\'e inequality}.
\end{thm}
As a consequence we obtain the following statement "of many zeros".
\begin{thm} For every $\rho$-lattice $\{x_{j}\}$ and the corresponding set of functions $\{\xi_{j}\}$ defined in (\ref{xi}) the following inequality holds for any $f\in \mathcal{D}(\mathcal{L}^{1/2})$ such that $ \left<f, \xi_{j}\right>=0$ for all $j$:
\begin{equation}\label{zeros-ineq}
\|f\|^{2}\leq CN(\mathbf{X})\rho^{2}\|\mathcal{L}^{1/2}f\|^{2},
\end{equation}
where $C$ is the same as in (\ref{Poinc-0}).
\end{thm}
To obtain another important consequence of Theorem \ref{GPI} we need the following lemma which was proved in \cite{Pes00}.
\begin{lem}\label{key}
 If  $T$ is a self-adjoint operator in a Hilbert space
 and for some $\varphi$ in the
domain of $T$
$$
\|\varphi\|\leq A+ a\|T\varphi\|, \>\>\>a>0,
$$
then for all $k=2^{l}, \>l=0,1,2, ...$
\begin{equation}\label{key2}
\|\varphi\|\leq kA+ 8^{k-1}a^{k}\|T^{k}\varphi\|,
\end{equation}
as long as $\varphi$ belongs to the domain of $T^{k}$.
\end{lem}

\begin{thm}\label{graph-norm} For every $\rho$-lattice $X_{\rho}$ and every $m=2^{l}, l=0,1,2, ...$ there exist constants $c_{1}=c_{1}({\bf X}, \rho, m),  C_{1}=C_{1}({\bf X})$ such that 
\begin{equation}\label{equivalence}
c_{1}\left(\|f\|^{2}+\|\mathcal{L}^{m/2} f\|^{2}\right)^{1/2}\leq \left( \sum_{j}| \left<f, \xi_{j}\right>|^{2}+\|\mathcal{L}^{m/2} f\|^{2}\right)^{1/2}\leq
$$
$$
 C_{1}\left(\|f\|^{2}+\|\mathcal{L}^{m/2} f\|^{2}\right)^{1/2}.
\end{equation}

\end{thm}
\begin{proof} 
Consider (\ref{GPin}) with $\alpha=1$ and apply Lemma  \ref{key}. It will show existence of a $\widetilde{C}=\widetilde{C}({\bf X}, \rho, m)$ such that 
$$
\|f\|^{2}\leq \widetilde{C}\left(\sum_{j}| \left<f, \xi_{j}\right>|^{2}+\|\mathcal{L}^{m/2} f\|^{2}\right),
$$
which implies left side of  (\ref{equivalence}).
By using the H\"older  inequality we obtain 
\begin{equation}\label{norm-equivalence}
\sum_{j}|\left<f, \xi_{j}\right>|^{2}= \sum_{j}\left||B(x_{j}, \rho)|^{-1/2}\int_{B(x_{j}, \rho)}f\right|^{2}\leq
$$
$$
  \sum_{j}\int_{B(x_{j}, \rho)}|f|^{2}\leq N({\bf X})\|f\|^{2},
\end{equation}
where $N({\bf X})$ is from Lemma \ref{covering}. This inequality implies the right-hand side of (\ref{equivalence}). Theorem is proved.

\end{proof}

\section{ Frames of averages in Paley-Wiener spaces}
\subsection{Paley-Wiener functions  in $L_{2}(\mathbf{X})$.}

Since $\mathcal{L}$ is a self-adjoint non negative  definite operator
 in the Hilbert space $L_{2}(\mathbf{X})$ it has a uniquely defined non negative self- adjoint square root $\sqrt{\mathcal{L}}$. By using Spectral Theorem for $\sqrt{\mathcal{L}}$ and associated operational calculus one can define the projector ${\bf 1}_{[0,\>\omega]}(\sqrt{\mathcal{L}})$, where ${\bf 1}_{[0,\>\omega]}$ is the characteristic function of the interval $[0,\omega]$.

\begin{defn}
We will  say that a function $f \in L_{2}(\mathbf{X})$ belongs to the {\it Paley-Wiener space} $PW_{\omega}(\mathcal{L})$
if it belongs to the range of the projector ${\bf 1}_{[0,\>\omega]}(\sqrt{\mathcal{L}})$.

\end{defn}
Next we denote by $\mathcal{H}^{k}$  the domain  of $\mathcal{L}^{k/2}$.
It is a Banach space,
equipped with the graph norm $\|f\|_{k}=\|f\|+\|\mathcal{L}^{k/2}f\|$.
The next theorem contains generalizations of several results
from  classical harmonic analysis (in particular  the
Paley-Wiener theorem). It follows from our  results in
\cite{Pes00} and \cite{Pes09a}.

\begin{thm}The spaces $PW_{\omega}(\mathcal{L})$ have the following properties:
\begin{enumerate}
\item  the space $PW_{\omega}(\mathcal{L})$ is a linear closed subspace in
$L_{2}(\mathbf{X})$.
\item the space 
 $\bigcup _{ \omega >0}PW_{\omega}(\mathcal{L})$
 is dense in $L_{2}(\mathbf{X})$;

\item (Bernstein inequality)   $f \in PW_{\omega}(\mathcal{L})$ if and only if
$ f \in \mathcal{H}^{\infty}=\bigcap_{k=1}^{\infty}\mathcal{H}^{k}$,
and the following Bernstein inequalities  hold true
\begin{equation}\label{Bern100}
\|\mathcal{L}^{s/2}f\|\leq \omega^{s}\|f\| \quad \mbox{for all} \, \,  s\in \mathbb{R}_{+};
\end{equation}

\end{enumerate}
\end{thm}

\subsection{Sampling and Hilbert frames}\label{sampling/frames}

 A set of vectors $\{\theta_{v}\}$  in a Hilbert space $H$ is called a Hilbert  frame if there exist constants $a,b >0$ such that for all $f\in H$ 
\begin{equation}\label{Frame ineq}
a\|f\|^{2}_{2}\leq \sum_{v}\left|\left<f,\theta_{v}\right>\right|^{2}     \leq b\|f\|_{2}^{2}.
\end{equation}
The largest $a$ and smallest $b$ are called respectively the lower and the upper frame bounds and the ratio $b/a$ is known as the tightness of the frame.  If $a=b$ then $\{\theta_{v}\}$ is a \textit{tight } frame, and if $a=b=1$ it is called a \textit{Parseval } frame. 
Parseval frames are similar in many respects to orthonormal bases.  For example, if all members of a Parseval frame are unit vectors then it  is an  orthonormal basis. 

According to the general theory of Hilbert frames \cite{DS}, \cite{Gr} the frame  inequality (\ref{Frame ineq})  implies that there exists a dual frame $\{\Theta_{v}\}$ (which is not unique in general) 
for which the following reconstruction formula holds 
\begin{equation}\label{frame/reconstruction}
f=\sum_{v}\left<f,\theta_{v}\right>\Theta_{v}.
\end{equation}

In general  it is not easy to find a dual frame. For this reason one can resort 
to the following  frame algorithm (see  \cite{Gr}, Ch. 5) which performs reconstruction by iterations.       
Given a relaxation parameter $0<\mu<\frac{2}{b}$, set $\eta=\max\{|1-\mu a|,\> |1-\mu b|\}<1$. Let $f_{0}=0$ and define recursively
\begin{equation}
\label{rec}
f_{n}=f_{n-1}+\mu \Phi(f-f_{n-1}),
\end{equation}
where $\Phi$ is the frame operator which is defined on $H$ by the formula
$
\Phi f=\sum_{\nu}\left<f,\xi_{\nu}\right>\xi_{\nu}.
$
In particular, $f_{1}=\mu \Phi f=\mu\sum_{j}\left<f, \xi_{\nu}\right> \xi_{\nu}$. Then  $\lim_{n\rightarrow \infty}f_{n}=f$ with a geometric rate of convergence, that is, 
\begin{equation}
\label{conv}
\|f-f_{n}\|\leq \eta^{n}\|f\|.
\end{equation}
In particular,  for the choice $\mu=\frac{2}{a+b}$ the convergence factor is 
 $$ 
\eta=\frac{b-a}{a+b}.
$$
\subsection{Frames of averages in Paley-Wiener spaces}

\begin{thm}
For a given $\alpha>0$ there exists a constant $0<\nu=\nu({\bf X}, \alpha)<1$  such that for
any $\omega>0$, every metric $\rho$-lattice $\mathbf{X}_{\rho}=\{x_{j}\}$ with
$\rho= \nu\omega^{-1/2}$ the corresponding set of   functionals 
\begin{equation}\label{avfunc}
f\rightarrow \left<f, \xi_{j}\right>=|B(x_{j}, \rho)|^{-1/2}\int_{B(x_{j}, \rho)} f
\end{equation}
is a frame in $PW_{\omega}(\mathcal{L})$, i. e.  
\begin{equation}
\frac{\alpha}{2+2\alpha}\|f\|^{2}\leq \sum_{j}|\left<f, \xi_{j}\right>|^{2}\leq N(\mathbf{X})\|f\|^{2},\>\>\>\alpha>0.
\end{equation}

\end{thm}
\begin{proof}

Using  Bernstein inequality (\ref{Bern100}) and global Poincar\'e inequality (\ref{GPin}) one obtains  for $f\in PW_{\omega}(\mathcal{L})$
\begin{equation}\label{upper-ineq}
\|f\|^{2}\leq (1+\alpha) CN(\mathbf{X})(\rho\omega^{1/2})^{2}\|f\|^{2}+  \frac{1+\alpha}{\alpha}  \sum_{j}| \left<f, \xi_{j}\right>|^{2}.
\end{equation}
If for a fixed $\alpha>0$ one picks 
$$
\nu=\frac{1} {\sqrt{2(1+\alpha)CN(\mathbf{X})}  }
$$ 
then  by choosing $\rho$ which satisfies 
$
\rho=\nu\omega^{-1/2}
$
we can move the first term on the right in (\ref{upper-ineq}) to the left side to obtain 
\begin{equation}\label{last-ineq}
\|f\|^{2}\leq \left(\frac{2}{\alpha}+2\right)\sum_{j}| \left<f, \xi_{j}\right>|^{2},\>\>\>\alpha>0.
\end{equation}
Opposite inequality follows from (\ref{norm-equivalence}).
Theorem is proven.

\end{proof}

	\begin{col}
	
For a given $\alpha>0$ there exists a constant $0<\nu=\nu({\bf X}, \alpha)<1$  such that for
any $\omega>0$ and  every metric $\rho$-lattice $\mathbf{X}_{\rho}=\{x_{j}\}$ with
$\rho= \nu\omega^{-1/2}$ every function $f$  in $PW_{\omega}({\bf X})$ is completely determined by its averages (\ref{avfunc}) and can be reconstructed from them in a stable way by using formula (\ref{frame/reconstruction}). One can also use frame algorithm described by (\ref{rec}).
	\end{col}

\section{Average variational splines on metric measure spaces}\label{splines}

\subsection{Construction of average variational splines}
We consider a sufficiently small $\rho>0$, a fixed $\rho$-lattice $\{x_{j}\}$ and  the corresponding cover $\left\{B(x_{j},\>\rho)\right\}$. 
Let 
$$
 \mathcal{A}_{j}(f)=\mathcal{A}_{B(x_{j}, \rho)}(f)=\left<f,\xi_{j}\right>=|B(x_{j}, \rho)|^{-1/2}\int_{B(x_{j}, \rho)}f
 $$ 
 where $\xi_{j}=|B(x_{j}, \rho)|^{-1/2}\chi_{j}(x),$
  $\>\>\chi_{j}$ being the  characteristic function of $B(x_{j}, \rho)$.
Since all lattices have a uniform multiplicity $N({\bf X})$ one has that if $f\in L_{2}({\bf X})$ then $\{\mathcal{A}_{j}(f)\}\in l_{2}$:
$$
\sum_{j}\left |\mathcal{A}_{j}(f)\right|^{2}\leq \sum_{j}\int_{B(x_{j}, \rho)}|f|^{2}\leq N(\mathbf{X})\|f\|^{2}
$$

\begin{defn}
 For a  sufficiently large $k\in \mathbb{N}$ and a sequence ${\bf s}=\{s_{j}\}_{j\in J}\in l_{2}$ the set of all functions
in  $\mathcal{D}(\mathcal{L}^{k/2})\subset L_{2}(\mathbf{X})$ such
 that $\mathcal{A}_{j}(f)=s_{_{j}}$ will be denoted by $Z_{{\bf s}}^{k/2}$. In particular, the subspace $Z_{{\bf 0}}^{k/2}$ corresponds to the trivial sequence ${\bf 0}=\{0,0,... ,\}$.
 \end{defn}
 
  We introduce 
 $
 \mathcal{P}_{k/2}: \mathcal{D}(\mathcal{L}^{k/2})\rightarrow l_{2}
 $
 defined as 
 $$
 \mathcal{P}_{k/2}f=\left \{\mathcal{A}_{j}(f)\right\}\in l_{2},\>\>\>\>f\in \mathcal{D}(\mathcal{L}^{k/2}).
 $$
  In general, every map $ \mathcal{P}_{k/2}$ depends on $k$ and on a lattice. To simply our framework we adopt the following assumption which holds true in many important cases.
 \begin{assump}\label{ass3}
 There exists a $k_{0}\in \mathbb{N}$ such that for every $k>k_{0}$ and every lattice with sufficiently small $\rho>0$ the image of $\mathcal{D}(\mathcal{L}^{k/2})\subset L_{2}(\mathbf{X})$  under
 $
 \mathcal{P}_{k/2}
 $
 is the entire $l_{2}$.
 \end{assump}
 We consider the following optimization problem:
 
 \bigskip

 \textit{ For a given lattice $\{x_{j}\}$ and a sequence ${\bf s}=\{s_{j}\}_{j\in J}\in   l_{2}$ find a function $f$ in 
 $Z_{{\bf s}}^{k/2},\>\>\>k>k_{0},$ which minimizes the functional}
 
\begin{equation}\label{functional}
u\rightarrow \|\mathcal{L}^{k/2}u\|.
\end{equation}

\bigskip

\begin{thm}\label{unique}
If the  {\bf Assumptions \ref{ass1}, \ref{ass3}} are  satisfied then the optimization problem has a unique 
solution for every sequence ${\bf s}=\{s_{j}\}_{j\in J}\in  l_{2}$  and every $k>k_{0}$.
\end{thm}
\begin{proof} 

The same problem for functional 
\begin{equation}\label{norm-1}
u\rightarrow \left(\|u\|^{2}+\|\mathcal{L}^{k/2}u\|^{2}\right)^{1/2}, \>\>k>k_{0},\>\>u\in \mathcal{D}(\mathcal{L}^{k/2}),
\end{equation}
 can be solved easily.
For  a  given sequence $\{s_{j} \}\in  l_{2}$ consider a function $f$
in  $\mathcal{D}(\mathcal{L}^{k/2})$ such that $\mathcal{A}_{j}(f)=s_{_{j}}.$ Such function exists due to the last assumption. Let $Pf$
 denote the orthogonal projection of this function $f$  in the Hilbert
space $\mathcal{D}(\mathcal{L}^{k/2})$ with the natural inner product
\begin{equation}\label{inner-1}
\left<f,g\right>+\left<\mathcal{L}^{k/2}f,\>\mathcal{L}^{k/2}g\right>
\end{equation}
on the subspace
$Z_{\mathbf{0}}^{k/2}$ with the norm $ \left(\|u\|^{2}+\|\mathcal{L}^{k/2}u\|^{2}\right)^{1/2}$.
Then the function $g=f-Pf$ will be the unique solution of the
above minimization problem for the
 functional (\ref{norm-1}). Both existence and uniqueness follow from the fact that any two functions in $Z_{\mathbf{s}}$ are differ by an $h\in Z_{\mathbf{0}}^{k/2}$.
 Indeed, 
 \begin{enumerate}
 
 \item to show uniqueness notice that for any $f_{1}$ in $ Z_{{\bf s}}$ one has $f_{1}=f+h$, where $h\in Z_{{\bf 0}}$ and since $h=Ph$
 $$
 f_{1}-Pf_{1}=f+h-P(f+h)=f-Pf;
 $$

 \item  minimality follows from 
  $$
\|g+h\|^{2}+\|\mathcal{L}^{k/2}(g+h)\|^{2}=\left( \|g\|^{2}+ \|\mathcal{L}^{k/2}g\|^{2}\right)+\left( \|h\|^{2}+ \|\mathcal{L}^{k/2}h\|^{2}\right),
 $$
where orthogonality of $g$ to $Z_{\mathbf{0}}^{k/2}$ with respect to (\ref{inner-1}) was used. 

\end{enumerate}

The problem with functional $u\rightarrow \|\mathcal{L}^{k/2}u\|$ is that it is
not a norm.  But we already proved in Theorem \ref{graph-norm} that the inner product (\ref{inner-1}) is equivalent to the inner product 
\begin{equation}\label{inner-2}
\sum_{j}|\mathcal{A }_{j}(f)\mathcal{A}_{j}(g) |^{2}+\left<\mathcal{L}^{k/2}f,\>\mathcal{L}^{k/2}g\right>.
\end{equation}
Thus  the above procedure can be applied to the Hilbert space 
with
the inner product (\ref{inner-2}) 
 and it clearly proves existence and uniqueness of the solution of our
minimization problem for the functional $u\rightarrow \|\mathcal{L}^{k/2}u\|$. Theorem is proved. 
\end{proof}

\begin{defn}For the given ${\bf s}=\{s_{j}\}$ the corresponding unique solution will be called the variational spline and denoted  as $S_{k}({\bf s})$. 
\end{defn}

We have the following characterization of the
the above optimization problem.
\begin{thm} For a sequence $s=\{s_{j}\}_{j\in J}\in   l_{2}$
a function $f\in Z_{{\bf s}}^{k/2},\>\>\>k>k_{0},$ minimizes the functional (\ref{functional}) 
if and only if $\mathcal{L}^{k/2}f$ is orthogonal to $\mathcal{L}^{k/2}Z_{0}$ in $L_{2}({\bf X})$.
\end{thm}

\begin{proof}If $f\in Z_{{\bf s}}^{k/2},\>\>\>k>k_{0},$ then any other function in $Z_{{\bf s}}^{k/2}$ has the form $f+h$ for some $h\in Z_{{\bf 0}}$.
If $\mathcal{L}^{k/2}f$ is orthogonal to $\mathcal{L}^{k/2}Z_{0}$ then 
$$
\|\mathcal{L}^{k/2}(f+ h)\|^{2} = \|\mathcal{L}^{k/2}f\|^{2}+
\|\mathcal{L}^{k/2}h\|^{2},
$$
which shows that $f$ is the minimizer. 
Conversely, let $f\in Z_{{\bf s}}^{k/2},\>\>\>k>k_{0}$ and $f$ minimizes (\ref{functional}). For
any $h\in Z_{0},\>\>\lambda\in \mathbb{R}$ one has
$$\|\mathcal{L}^{k/2}(f+\lambda h)\|^{2}=\|\mathcal{L}^{k/2}f\|^{2}
+2\lambda  \langle \mathcal{L}^{k/2}f,\mathcal{L}^{k/2}h \rangle +\lambda^{2}\|\mathcal{L}^{k/2}h\|^{2}.
$$
It shows that the vector $f$ can be a minimizer only if $ \langle \mathcal{L}^{k/2}f,\mathcal{L}^{k/2}h \rangle
=0$ since otherwise the minimizer would be given by the formula 
$$
f-\frac{\langle\mathcal{L}^{k/2}f,\mathcal{L}^{k/2}h\rangle}{\|\mathcal{L}^{k/2}f\|^{2}}h.
$$
\end{proof}

\subsection{Interpolation and approximation by average splines}
\begin{defn}
Consider a sufficiently small $\rho>0$, a fixed $\rho$-lattice $\{x_{j}\}$ and  the corresponding cover $\left\{B(x_{j},\>\rho)\right\}$. 
For $f\in E_{k},\>\>k>k_{0},$ we denote by $S_{k}(f)$  the solution of the
minimization problem such that $S_{k}(f)-f\in Z_{0}.$ We say that $S_{k}(f)$ is a variational spline which interpolates $f$ by its average values on balls  $\left\{B(x_{j},\>\rho)\right\}$. 

\end{defn}

Note, that this terminology is justified since $S_{k}(f)-f\in Z_{0}$ if and only if for every $j$
$$
\int_{B(x_{j}, \rho)}S_{k}(f)=\int_{B(x_{j}, \rho)}f.
$$
The following Lemma was proved in \cite{Pes98a}, \cite{Pes01}.
\begin{lem}\label{lemma}
If $T$ is a self-adjoint  non-negative operator in a
Hilbert space $H$ and for an  $\varphi\in H$ and a positive $a>0$
the following inequality holds true
$$
\|\varphi\|\leq a\|T\varphi\|,
$$ then for the same $\varphi \in H$, and all $ k=2^{l}, l=0,1,2,...$ the following
inequality holds
\begin{equation}\label{lemma2}
\|\varphi\|\leq a^{k}\|T^{k}\varphi\|,
\end{equation}
as long as $\varphi$ in the domain of $T^{k}$.
\end{lem}

\begin{proof}
By the spectral theory \cite{BS} there exist a direct integral of
Hilbert spaces
$$
X=\int_{0}^{\infty} X(\tau)dm (\tau )
$$
 and a unitary operator
$F$ from $H$ onto $X$, which transforms domain of $T^{t}, t\geq
0,$ onto $X_{t}=\{x\in X|\tau^{t}x\in X \}$ with norm
$$
\|T^{t}f\|_{H}=\left (\int_{0}^{\infty} \tau^{2t} \|Ff(\tau
)\|^{2}_{X(\tau)} d m(\tau) \right )^{1/2}
$$
and $F(T^{t} f)=\tau ^{t} (Ff)$. According to our assumption we
have for a particular  $\varphi\in H$
$$
\int _{0}^{\infty}| F\varphi(\tau)|^{2}d m(\tau)\leq a^{2} \int
_{0}^{\infty}\tau ^{2}| F\varphi(\tau)|^{2}d m(\tau)
$$
and then for the interval $\mathcal{I}=\mathcal{I}(0, a^{-1})$ we have
$$\int _{\mathcal{I}}| F\varphi(\tau)|^{2}d m(\tau)+
\int_{[0, \infty]\setminus \mathcal{I}}|F\varphi|^{2}d m(\tau)\leq
$$
$$
a^{2}\left( \int_{\mathcal{I}}\tau^{2}|F\varphi|^{2}d m(\tau) +\int_{[0,
\infty]\setminus \mathcal{I}}\tau^{2}|F\varphi| ^{2}d m(\tau)\right ) .
$$
 Since $a^{2}\tau^{2}<1$ on $\mathcal{I}(0, a^{-1})$
$$
0\leq \int_{\mathcal{I}}\left
(|F\varphi|^{2}-a^{2}\tau^{2}|F\varphi|^{2}\right)d m(\tau) \leq
\int _{[0, \infty]\setminus \mathcal{I}}\left ( a^{2}
\tau^{2}|F\varphi|^{2}-|F\varphi|^{2}\right)d m(\tau).
$$
This inequality
 implies the inequality
 $$
 0\leq \int_{\mathcal{I}}\left (a^{2}\tau^{2}|F\varphi|^{2}-
a^{4}\tau^{4}|F\varphi|^{2}\right)d m(\tau) \leq \int_{[0,
\infty]\setminus \mathcal{I}}\left ( a^{4}
\tau^{4}|F\varphi|^{2}-a^{2}\tau^{2}|F\varphi| ^{2}\right)d
m(\tau)$$
 or
  $$a^{2}\int_{0}^{\infty}\tau^{2}|F\varphi|^{2}d m(\tau) \leq
 a^{4}\int_{0}^{\infty}\tau^{4}|F\varphi|
 ^{2}d m(\tau),
 $$
 which means
 $$
 \|\varphi\|\leq a\|T\varphi\|\leq a^{2}\|T^{2}\varphi\|.
 $$
 Now, by using induction one can finish the proof of the
Lemma. The Lemma is proved.
\end{proof}
\begin{rem}
By using Lemma \ref{key} with $A=0$ one could have (\ref{key2}) with $A=0$ however, the inequality (\ref{lemma2}) is  stronger.

\end{rem}

\begin{thm}
For every $k=2^{l},\>\>l\in \mathbb{N},$ and every $f\in \mathcal{D}(\mathcal{L}^{k/2})$  the following inequality takes place
\begin{equation}\label{approx}
\|f-S_{k}(f)\|^{2}\leq 4C_{0}^{k}\rho^{2k}\|\mathcal{L}^{k/2}f\|^{2}.
\end{equation}
In particular, if $f\in PW_{\omega}(\mathcal{L})$ then
\begin{equation}\label{approx-PW}
\|f-S_{k}(f)\|^{2}\leq (C_{1}\rho\omega)^{2k}\|f\|^{2},
\end{equation}
where $C_{1}=2^{1/k}C_{0}^{1/2}<2C_{0}^{1/2}$.
\end{thm}

\begin{proof}
If $k=2^{l}$ then since $f-S_{k}(f)$ belongs to $Z_{0}$  by 
(\ref{zeros-ineq})  we have
$$
\|f-S_{k}(f)\|^{2}\leq CN({\bf X})\rho^{2}\|\mathcal{L}^{1/2}(f-S_{k}(f))\|^{2}.
$$
Using Lemma \ref{lemma}  with $T=\mathcal{L}^{1/2}$ we obtain
$$
\|f-S_{k}(f)\|^{2}\leq C_{0}^{k}\rho^{2k}\|\mathcal{L}^{k/2}(f-S_{k}(f))\|^{2},
$$
with $C_{0}=CN({\bf X})$ and the minimality  of
$S_{k}(f)$ gives
$$
\|f-S_{k}(f)\|^{2}\leq 4C_{0}^{k}\rho^{2k}\|\mathcal{L}^{k/2}f\|^{2}.
$$
Combining this inequality with the Bernstein inequality (\ref{Bern100}) 
we are coming to (\ref{approx-PW}). Theorem is proven.

\end{proof}
As a summary we have the following statement.
\begin{thm}
Assume that the  {\bf Assumptions \ref{ass1},  \ref{ass3}} are  satisfied.  Consider $ PW_{\omega}(\mathcal{L}),\>\omega>0,$  and $\rho<(C_{1}\omega)^{-1},$ where $C_{1}$ is from (\ref{approx-PW}).  For every $\rho$-lattice $\{x_{j}\}$ and the corresponding set of functions $\{\xi_{j}\}$ defined in (\ref{xi}) every $f\in  PW_{\omega}(\mathcal{L}),\>\omega>0,$ is uniquely determined by the set of values $\{\langle f, \xi_{j}\rangle\}$ and can be reconstructed by the formula
\begin{equation}\label{approx-PW-100}
\lim_{k\rightarrow \infty}S_{k}(f)=f,
\end{equation}
where the rate of convergence is
\begin{equation}
\|f-S_{k}(f)\|\leq \gamma^{k}\|f\|,
\end{equation}
with $\gamma=C_{1}\rho\omega<1$.

\end{thm}
\section{Average sampling and  average splines  on combinatorial graphs}

\subsection{Analysis on Graphs}

\label{sect:main}

Let $G$ denote an undirected weighted graph, with  a finite or countable number of vertices $V(G)$  and weight function $w: V(G) \times V(G) \to \mathbb{R}_0^+$. $w$ is symmetric, i.e., $w(u,v) = w(v,u)$, and $w(u,u)=0$ for all $u,v \in V(G)$. The edges of the graph are the pairs $(u,v)$ with $w(u,v) \not= 0$.

 Our assumption is that for every $v\in V(G)$ the following finiteness condition holds
\begin{equation}\label{cond:finiteness}
 w(v) = \sum_{u \in V(G)} w(u,v)<\infty.
\end{equation}

Let $L_{2}(G)\>\>$ denote the space of  all complex-valued functions with the inner product
$$
\left<f,g\right>=\sum_{v\in V(G)}f(v)\overline{g(v)}
$$
and  the norm
\[
 \| f \|_{2} = \left( \sum_{v \in V(G)} |f(v)|^2 \right)^{1/2}.
\]

\begin{defn}The  weighted gradient norm of a function $f$ on $V(G)$ is defined by 
\[
 \| \nabla f \| = \left( \sum_{u, v \in V(G)} \frac{1}{2} |f(u) - f(v)|^2 w(u,v) \right)^{1/2}.
\] 
 The set of all $f: G \to \mathbb{C}$ for which the weighted gradient norm is finite will be denoted as $\mathcal{D}^{2}(\nabla)$.
\end{defn}

\begin{rem}
The factor $\frac{1}{2}$ makes up for the fact that every edge (i.e., every {\em unordered} pair $(u,v)$) enters twice in the summation. Note also that loops, i.e. edges of the type $(u,u)$, in fact do not contribute.

\end{rem}

We intend to prove Poincar\'e-type estimates involving weighted gradient norm.

In  the case of a \textit{finite} graph and $L_{2}(G)$-space  the weighted Laplace operator $\mathcal{L}: L_{2}(G) \to L_{2}(G)$  is introduced  via
\begin{equation}\label{L}
 (\mathcal{L}f)(v) = \sum_{u \in V(G)} (f(v)-f(u)) w(v,u)~.
\end{equation}
This  graph Laplacian is a well-studied object; it is known to be a positive-semidefinite self-adjoint \textit{bounded} operator.

According to  Theorem 8.1 and Corollary 8.2 in \cite{H}  if for an \textit{infinite} graph    there exists a $C>0$ such that the degrees are uniformly bounded
\begin{equation}\label{C}
 w(u) = \sum_{u \in V(G)} w(u,v)\leq C
\end{equation}
then operator which is defined by (\ref{L}) on functions with compact supports has a unique  positive-semidefinite self-adjoint \textit{bounded} extension $\mathcal{L}$ which is acting according to  (\ref{L}).

What is really important for us is that in both of these cases  for  the non-negative square root $\mathcal{L}^{1/2}$ one has the equality
\begin{equation}\label{L-G}
\|\mathcal{L}^{1/2}f\|=\|\nabla f  \|
\end{equation}
for all  $f\in\mathcal{D}^{2}(\nabla)$. This fact is not difficult to show directly (see \cite{Mo}, \cite{FP}).
\begin{lem}\label{grad-laplace}
For all $f \in L_{2}(G)$ contained in the domain of $\mathcal{L}$, we have
\begin{equation}\label{GL}
  \| \mathcal{L}^{1/2} f \|^2 = \| \nabla f \|^2~.
\end{equation}
 For  $f \in PW_{\omega}(\mathcal{L})$, this implies 
\begin{equation}
\label{Bern}
\|\nabla f\|=\|\mathcal{L}^{1/2}f\|\leq \sqrt{\omega}\|f\|.
\end{equation}
\end{lem}
\begin{proof}
We obtain 
\begin{eqnarray*}
 \langle f, \mathcal{L} f \rangle & = &  \sum_{u \in V(G)} f(u) \overline{ \left( \sum_{v \in V(G)} \left( f(u) - f(v) \right) w(u,v) \right)} \\
 & = & \sum_{u \in V(G)} \left( |f(u)|^2 w(u) - \sum_{v  \in V(G)} f(u) \overline{f(v)} w(u,v) \right) ~.
\end{eqnarray*}
In the same way 
\begin{eqnarray*}
 \langle f, \mathcal{L} f \rangle & = & \langle \mathcal{L} f, f \rangle \\
 & = & \sum_{u \in V(G)} \left( |f(u)|^2 w(u) - \sum_{v  \in V(G)} \overline{f(u)} f(v) w(u,v) \right)~.
\end{eqnarray*}
Averaging these equations yields
\begin{eqnarray*}
 \langle f, \mathcal{L} f \rangle &  = &  \sum_{u \in V(G)} \left( |f(u)|^2 d(u) - {\rm Re} \sum_{v  \in V(G)} f(u) \overline{f(v)} w(u,v) \right)  \\
& = & \frac{1}{2} \sum_{ u,v \in V(G)} |f(u)|^2 w(u,v) + |f(v)|^2 w(u,v) - 2  {\rm Re}  f(u) \overline{f(v)} w(u,v) \\
& = & \sum_{u,v  \in V(G)} \frac{1}{2} |f(v)-f(u)|^2 w(u,v)  =  \| \nabla f \|^2~.
\end{eqnarray*}
Now the first equality follows by taking the non-negative square root of $\mathcal{L}$ (note that by spectral theory, $f$ is also in the domain of $\mathcal{L}^{1/2}$, and (\ref{Bern}) is an obvious consequence.
\end{proof}

\section{A Poincar\'e-type inequality for finite graphs}
 
 It is well known that for every finite connected graph has $\lambda_{0}=0$  as a simple eigenvalue of the Laplace operator $L$ and the corresponding eigenfunction is a constant on the entire graph.
 Given a connected and finite graph $G$  and a function $f\in L_{2}(G)$ we consider its average
 $$
 f_{G}=\frac{1}{|G|}\sum_{v\in V(G)}f(v).
 $$
The notation $a\mathbf{1}$ is used for a constant function $f(v)=a$ for all $v\in G$.

 \begin{thm} \label{Poinc}
 For every  connected and finite graph $G$ (which contains more than one vertex)  the following Poincar\'e-type inequality holds
 \begin{equation}\label{Poinc-G}
\sum_{u\in V(G)}\left|f(u)-f_{G}\mathbf{1}\right|^{2}\leq\frac{1}{\lambda_{1}}\|\nabla f\|^{2}=\frac{1}{\lambda_{1}}\|\mathcal{L}^{1/2}f\|^{2},\>\>\>\>f\in L_{2}(G),
 \end{equation}
 where $\lambda_{1}$ is the first non-zero eigenvalue of $\mathcal{L}$. 
 \end{thm}

\begin{proof}
Note, that  the average of the function $f-f_{G}\mathbf{1}$  is zero:
$$
\sum_{u\in V(G)} \left(f(u)-\left(\frac{1}{|G|}\sum_{v\in V(G)}f(v)\right)\mathbf{1}\right)=
\sum_{u\in V(G)}f(u)-\sum_{v\in V(G)}f(v)=0.
$$
If $\lambda_{1}$  is the first nonzero eigenvalue  of $\mathcal{L}$ then $\sqrt{\lambda_{1}}$ is the first nonzero eigenvalue of the nonnegative  square root $\mathcal{L}^{1/2}$.
 Since function $f-f_{G}{\bf 1}$ is orthogonal to constants it implies 
\begin{equation}
\|f-f_{G}{\bf 1}\|\leq \frac{1}{\sqrt{\lambda_{1}}}\|\mathcal{L}^{1/2}(f-f_{G}{\bf 1})\|=\frac{1}{\sqrt{\lambda_{1}}}\|\mathcal{L}^{1/2}f\|.
\end{equation}
But according to Lemma \ref{grad-laplace} it gives 
$$
\|f-f_{G}{\bf 1}\|\leq\frac{1}{\sqrt{\lambda_{1}}}\|\nabla f\|.
$$
Theorem is proven.
\end{proof}

\section{A Local and Global   Poincare-type inequalities for finite and infinite graphs}

Let $G$ be a finite or infinite and countable connected graph and $\Omega\subset V(G)$ is a finite and connected subset of vertices which we will treat as an {\bf induced} graph  and will  denote by the same letter $\Omega$. 
We remind that this  means that the set of vertices of such graph, which will be denoted as $V(\Omega)$, is exactly the set of vertices in $\Omega$ and the set of edges are all edges in $G$ whose both ends belong to $\Omega$.
Let $\Delta_{\Omega}$ be the Laplace operator constructed according to (\ref{L})  for  such induced graph $\Omega$. The first nonzero eigenvalue of the operator operator $\Delta_{\Omega}$ will be denoted as $\lambda_{1, \Omega}$. Let $w_{\Omega}(u,v),\>\>u,v\in V( \Omega),$ and 
$$
w_{\Omega}(v)= \sum_{u \in V(\Omega)} w_{\Omega}(u,v),\>\>v\in V(\Omega),
$$ 
be the corresponding weight functions.
We notice that for every $\Omega$ and every $u, v\in V(\Omega)$ one has $w(u,v)=w_{\Omega}(u,v)$. However, in general $w(u)\geq w_{\Omega}(u)$.

Suppose that $\Xi=\{\Omega_{j}\}$ is  a disjoint cover of $V(G)$ by connected and finite subgraphs $\Omega_{j}$. We define functions $\xi_{j}$ by the formula  
$$
\xi_{j}=\frac{1}{\sqrt{ |\Omega_{j}|}}\chi_{j},
$$ 
where  $\chi_{j}$ is the characteristic function of $\Omega_{j}$, and $|\Omega_{j}|$ is the number of vertices in $\Omega_{j}$. We will be interestead  in functionals on $L_{2}(G)$ defined by these functions
$$
f\mapsto \langle f, \xi_{j} \rangle=\frac{1}{\sqrt{ |\Omega_{j}|}}\sum_{v\in V(\Omega_{j})}f(v),\>\>\>\>f\in L_{2}(G).
$$
 We will also need functions
$$
\zeta_{j}=\frac{1}{|\Omega_{j}|}\chi_{j},
$$
and corresponding functionals
$$
f\mapsto \langle f, \zeta_{j} \rangle=\frac{1}{ |\Omega_{j}|}\sum_{v\in V(\Omega_{j})}f(v),\>\>\>\>f\in L_{2}(G).
$$

 By using these notations we formulate the next two theorems.
Our  {\bf local Poincar\'e-type inequality} is the following.
\begin{thm}\label{Local-Poinc}
Let $G$ be a finite or infinite and countable graph. Let $\Omega\subset V(G)$ be a finite {\bf induced subgraph}.  Let $\Delta_{\Omega}$  be the Laplace operator of the  induced  graph $\Omega$ whose first nonzero eigenvalue is $\lambda_{1,\Omega}$. Then for every $f\in L_{2}(G)$ the following inequality holds
\begin{equation}\label{Poinc-g}
\sum_{u\in V(\Omega)}|f(u)-f_{\Omega}\chi_{\Omega}(u)|^{2}\leq
  \frac{1}{\lambda_{1, \Omega}}\|\nabla_{\Omega} f\|^{2}_{\Omega},
 \end{equation}
 where 
 $$
 \left\| \nabla_{\Omega} f \right\|^{2}_{\Omega} =  \sum_{u, v \in V(\Omega)} \frac{1}{2} \left|f(u) - f(v)\right|^2 w(u,v).
 $$
 \end{thm}
Clearly, it is a direct consequence of Theorem \ref{Poinc}.
The next Theorem contains what we call  {\bf a global Poincar\'e-type inequality}.

\begin{thm}\label{Main-Th}
Let $G$ be a connected finite or infinite and countable graph. Suppose that $\Xi=\{\Omega_{j}\}$ is  a disjoint cover of $V(G)$ by connected and finite subgraphs $\Omega_{j}$.  Let $\Delta_{j}$  be the Laplace operator of the  {\bf induced}  graph $\Omega_{j}$ whose first nonzero eigenvalue is $\lambda_{1,j}$. We assume that 
 that there exists a non zero lower boundary of all $\lambda_{1,j}$:
 $$
 \Lambda=\Lambda_{\Xi}=\inf_{j}\lambda_{1,j}>0.
 $$
In these  notations the following inequality holds for every $f\in L_{2}(G)$ and every $\alpha>0$

\begin{equation}\label{main-ineq}
\|f\|^{2}\leq \frac{1+\alpha}{\alpha}\frac{1}{\Lambda_{\Xi}} \|\mathcal{L}^{1/2}f\|^{2}+(1+\alpha)\sum_{j}\left|\left<f, \xi_{j}\right>\right|^{2}.
\end{equation}
\end{thm}

\begin{proof}
One has
\begin{equation}
\|f\|^{2}=\sum_{v\in V(G)}|f(v)|^{2}=\sum_{j}\left(\sum_{v\in V( \Omega_{j})}|f(v)|^{2}\right).
\end{equation}
For every $u\in  V(\Omega_{j})$ we  apply  (\ref{in}) to obtain the next inequality in which $f_{\Omega_{j}}=\left<f, \zeta_{j}\right>$
\begin{equation}\label{algebra}
|f(u)|^{2}\leq \frac{1+\alpha}{\alpha}\left|f(u)-f_{\Omega_{j}}\chi_{j}(u)\right|^{2}+(1+\alpha)\left|f_{\Omega_{j}}\chi_{j}(u)\right|^{2},
\end{equation}
which holds for every  positive $\alpha>0$. 
By using (\ref{Poinc-g}) we obtain
\begin{equation}\label{GlobPoinc}
\sum_{u\in V(\Omega_{j})}|f(u)|^{2}\leq \frac{1+\alpha}{\alpha}\sum_{u\in V(\Omega_{j})}|f(u)-f_{\Omega_{j}}\chi_{j}(u)|^{2}+(1+\alpha)\sum_{u\in V(\Omega_{j})}\left|f_{\Omega_{j}}\chi_{j}(u)\right|^{2}\leq
$$
$$
  \frac{1+\alpha}{\alpha}\frac{1}{\lambda_{1, j}}\|\nabla_{j} f\|^{2}_{\Omega_{j}}+(1+\alpha)|\Omega_{j}|\left|f_{\Omega_{j}}\right|^{2}.
  \end{equation}
Summation over $j$ gives the following inequality
\begin{equation}
\|f\|^{2}\leq  \frac{1+\alpha}{\alpha}\frac{1}{\Lambda_{\Xi}}\sum_{j} \|\nabla_{j} f\|^{2}_{\Omega_{j}} +(1+\alpha)\sum_{j}|\Omega_{j}|\left|\left<f,\zeta_{j}\right>\right|^{2}=
$$
$$
\frac{1+\alpha}{\alpha}\frac{1}{\Lambda_{\Xi}}\sum_{j} \left(\sum_{u, v \in V(\Omega_{j})}\frac{1}{2}|f(u) - f(v)|^2 w_{j}(u,v)\right)+(1+\alpha)\sum_{j}\left|\left<f, \xi_{j}\right>\right|^{2},\>\>\>\alpha>0.
\end{equation}

Since for all $j$ one has that $w(u,v)=w_{j}(u,v),\>\>u,v\in V(\Omega_{j}),$ and since sets $\Omega_{j}$ are disjoint, it is obvious that the first term in the last line  is not greater than 
$$
\frac{1+\alpha}{\alpha}\frac{1}{\Lambda_{\Xi}}\|\nabla f\|^{2}.
$$
 It gives
\begin{equation}\label{result}
\|f\|^{2}\leq \frac{1+\alpha}{\alpha}\frac{1}{\Lambda_{\Xi}} \|\nabla f\|^{2}+(1+\alpha)\sum_{j}\left|\left<f, \xi_{j}\right>\right|^{2},\>\>\>\alpha>0,
\end{equation}
and by applying 
Theorem \ref{L} we obtain (\ref{main-ineq}).
Theorem is proved.
\end{proof}

 \begin{rem}\label{exp}
 We are making a list of differences between Poincar\'e inequalities we used on Dirichlet space and on combinatorial graphs.
 \begin{enumerate}

 \item The Poincar\'e inequality on Dirichlet spaces (\ref{Poinc-0}) is formulated for balls, while a local Poincar\'e-type inequality for graphs (\ref{Poinc-g}) is for {\bf any finite induced subgraph}.
 
 \item There is a radius squared on the right-hand side in  (\ref{Poinc-0}) and reciprocal of the first eigenvalue of the Laplacian on a corresponding induced subgraph on the right-hand side in (\ref{Poinc-g}). Note, that (\ref{Poinc-g}) is sharp.
  It obviously shows that if for every ball $B(x,r)$ one has
 $$
 \frac{1}{\lambda_{1,B(x,r)}}\leq Cr^{2},\>\>\>C>0,
 $$
 then our (\ref{Poinc-g}) implies a "regularly" looking inequality.

 \item  In fact, one can implement in case of graphs the formulas (\ref{Gamma}) and 
 (\ref{nabla}) to see that 
 $$
 \Gamma(f,f)(v)=\frac{1}{2}\sum_{u \in V(G)} |f(v)-f(u)|^2 w(u,v),
 $$
 and 
 $$
 \sum_{v\in V(G)}  \Gamma(f,f)(v)=\sum_{v\in V(G)}\left(\frac{1}{2}\sum_{u \in V(G)} |f(v)-f(u)|^2 w(u,v)\right)=\|\nabla f\|^{2}.
 $$
 In this case the integral $\int_{B(x,r)}\Gamma(f,f)d\mu$ on the right-hand side of (\ref{Poinc-0}) corresponds to the quantity
 $$ 
  \sum_{v\in \Omega}\left(\frac{1}{2}\sum_{u \in V(G)} |f(v)-f(u)|^2 w(u,v)\right) ,    
  $$
  which is clearly not smaller than $\|\nabla_{\Omega} f\|^{2}_{\Omega}$ which appears on the right-hand side of (\ref{Poinc-g}).

 \end{enumerate}
 
 \end{rem}

\section{A sampling theorem and a reconstruction methods using frames}
\begin{thm}\label{Main-Th2}
If the assumptions of the previous Theorem hold then the set of functionals $\{\zeta_{j}\}$ is a frame in any space $PW_{\omega}(\mathcal{L})$ as long as $\Lambda_{\Xi}>\frac{1+\alpha}{\alpha}\omega.$ In other words, if 
\begin{equation}\label{gamma}
\gamma=\frac{1+\alpha}{\alpha}\frac{\omega}{\Lambda_{\Xi}}<1,\>\>\>\alpha>0,
\end{equation}
then
\begin{equation}\label{P-P}
\frac{(1-\gamma)}{(1+\alpha)}\|f\|^{2}\leq \sum_{j}\left|\left<f, \xi_{j}\right>\right|^{2}\leq\|f\|^{2}.
\end{equation}
\end{thm}
\begin{proof}
Indeed, if $f\in PW_{\omega}(\mathcal{L})$ then by the Bernstein inequality the (\ref{main-ineq}) can be rewritten as 
$$
\|f\|^{2}\leq \frac{1+\alpha}{\alpha} \frac{\omega}{\Lambda_{\Xi}} \| f\|^{2}+(1+\alpha)\sum_{j}\left|\left<f, \xi_{j}\right>\right|^{2}.
$$
If (\ref{gamma}) holds then one has
$$
0<(1-\gamma)\|f\|^{2}\leq (1+\alpha)\sum_{j}\left|\left<f, \xi_{j}\right>\right|^{2}.
$$
On the other hand, since
$$
\left|\sum_{v\in V( \Omega_{j})}f(v)\right|^{2}\leq
\left|\Omega_{j}\right|  \left( \sum_{v\in V( \Omega_{j})}|f(v)|^{2}\right),
$$
one has 
$$
\sum_{j}|\left<f, \xi_{j}\right>|^{2}= 
\sum_{j}\frac{1}{|\Omega_{j}|}\left|\sum_{v\in V(\Omega_{j})}f(v)\right|^{2}\leq  
\sum_{j}      \left( \sum_{v\in V( \Omega_{j})}|f(v)|^{2}\right)\leq\|f\|^{2}.
$$
Theorem is proven.
\end{proof}

Note, that for the classical Paley-Wiener spaces on the real line the inequalities similar to (\ref{P-P}) in the case when $\{\xi_{j}\}$ are delta functions were proved by Plancherel and Polya. Today they are better  known as the frame inequalities.
Now we can formulate sampling theorem based on average values.

	\begin{thm} \label{conclusion} Under the same conditions and notations as above every function $f\in PW_{\omega}(\mathcal{L})$ is uniquely determined by its averages  
	$
	\left<f, \xi_{j}\right>
	$
 and can be reconstructed from this set of values in a stable way. 

\end{thm}

\subsection{Reconstruction algorithms in terms of  frames}

	What we just proved in the previous section is that 
	under the same assumptions as above the set of functionals $f\rightarrow \left<f, \xi_{j}\right>$ is a frame in the subspace $PW_{\omega}(\mathcal{L})$.
	This fact allows to apply formula(\ref{frame/reconstruction}) which describes  a stable method of
reconstruction of a function $f\in PW_{\omega}(\mathcal{L})$ from a
 set of samples $\{\left<f,\xi_{j}\right>\}$.
Another possibility for reconstruction is to use the  frame algorithm given bt (\ref{rec}). 

\section{Average Variational Splines and a reconstruction algorithm}

\subsection{Variational interpolating splines}
As in the previous sections we assume that $G$ is a connected finite or infinite and countable graph and $\Xi=\{\Omega_{j}\}$ is  a disjoint cover of $V(G)$ by connected and finite subgraphs $\Omega_{j}$.

  For a  given sequence $\mathbf{v}=\{v_{j}\}\in l_{2}$ the set of all functions 
in  $L_{2}(G)$ such
 that $\left<f, \xi_{j}\right>=v_{j}$ will be denoted by $Z_{\mathbf{v}}^{k/2}$. In particular, $Z_{\mathbf{0}}^{k/2}$ corresponds to the sequence of zeros. 
 We consider the following optimization problem:

 \textit{For a given sequence $\mathbf{v}=\{v_{j}\}\in l_{2}$ find a function $f$ in
the set $Z_{\mathbf{v}}^{k/2}\subset L_{2}(G)$  which minimizes the functional
$$
u\rightarrow \|\mathcal{L}^{k/2}u\|, \>\>\>\>u\in Z_{\mathbf{v}}^{k/2}.
$$}
Similarly to Theorem \ref{unique} one can prove the following.
\begin{thm}
Under the above assumptions the optimization problem has a unique 
solution for every $k$.
\end{thm}
\begin{proof}Using Theorem \ref{Main-Th} one can justify the following algorithm:
\begin{enumerate}
 
 \item Pick any function $f\in Z_{\mathbf{v}}^{k/2}$. 
 
\item  Construct  $P_{0}f$ where $P_{0}$ 
is  the orthogonal projection of $f$  onto
 $Z_{\mathbf{0}}^{k/2}$ with respect to the inner product
 $$
 \left<f,g\right>_{k}= \sum_{j}\left<f, \xi_{j}\right>\left<g, \xi_{j}\right>+ \langle
\mathcal{L}^{k/2}f,
\mathcal{L}^{k/2}g \rangle.
$$
\item 
The function $f-P_{0}f$ is the unique solution to the given optimization
problem.
\end{enumerate}
\end{proof}

\begin{defn}
For $f\in L_{2}(G)$ the interpolating variational spline is denoted by $S_{k}(f)$  and it is the solution of the
minimization problem such that $S_{k}(f)-f\in Z_{\mathbf{0}}^{k/2}.$
\end{defn}
One can easily prove  the following characterization of  variational splines
\begin{thm}
A function $u\in L_{2}(G)$ is a variational spline  
if and only if $\mathcal{L}^{k}u$ is orthogonal to $\mathcal{L}^{k}Z_{\mathbf{0}}^{k/2}$.
\end{thm}

\subsection{Reconstruction using splines}
By applying this Lemma and using the same reasoning as in the first part of the paper  one can prove the following  reconstruction theorem. Below we are keeping  notations of Theorem \ref{Main-Th2}.

\begin{thm}  If the assumptions of  Theorem \ref{Main-Th2}  are satisfied and in particular
$$
\gamma=\frac{1+\alpha}{\alpha}\frac{\omega}{\Lambda_{\Xi}}<1,\>\>\>\alpha>0,
$$
then any function $f$ in $PW_{\omega}(\mathcal{L}),\>\>\>\> \omega>0,$ can be reconstructed from a set of its averages $\{\left<f, \xi_{j}\right>\}$ using the
formula  
$$
f=\lim_{k\rightarrow\infty}S_{k}(f),\>\>\>\>k=2^{l},\>\>\>
l=0,1, ...,
$$
and the error estimate is
\begin{equation}\label{last}
\|f-S_{k}(f)\|\leq   2 \gamma^{k}\|f\|, \>\>\>\>k=2^{l},\>\>\>
l=0,1, ... .
\end{equation}

\end{thm}

\begin{proof}
Pick an $k=2^{l},\>\>l=0,1,2,....$ and apply  (\ref{main-ineq}) to the function $f-S_{k}(f)$. It gives for every $\alpha>0$:
$$
\|f-S_{k}(f)\| \leq \frac{1+\alpha}{\alpha}\frac{1}{\Lambda_{\Xi}}\|\mathcal{L}^{1/2}(f-S_{k}(f)\|^{2}.
$$
By Lemma \ref{lemma} it implies the following inequality 
$$
\left\|f-S_{k}(f)\right\|\leq \left( \frac{1+\alpha}{\alpha}\frac{1}{\Lambda_{\Xi}}\right)^{k} \|\mathcal{L}^{k/2}(f-S_{k}(f))\|^{2}.
$$
Using minimization property of $S_{k}(f)$ and the Bernstein inequality (\ref{Bern100}) for $f\in PW_{\omega}(\mathcal{L})$ one obtains (\ref{last}) with 
$$
\gamma=\frac{1+\alpha}{\alpha}\frac{\omega}{\Lambda_{\Xi}}.
$$
The assertion follows from the assumption that $\gamma<1$.
\end{proof}

\section{Example: Lattice $\mathbb{Z}$}

Let us  consider a one-dimensional infinite lattice $\mathbb{Z}=\{...,-1, 0, 1, ...\}$ as an inveighed graph.  The dual
group of the commutative additive group $\mathbb{Z}$ is the
one-dimensional torus. The corresponding Fourier transform
$\mathcal{F}$ on the space $L_{2}(\mathbb{Z})$   is defined by the
formula
$$
\mathcal{F}(f)(\xi)=\sum_{k\in \mathbb{Z}}f(k)e^{ik\xi}, f\in
L_{2}(\mathbb{Z}), \xi\in [-\pi, \pi).
$$
It gives a unitary operator from $L_{2}(\mathbb{Z})$ on the space
$L_{2}(\mathbb{T})=L_{2}(\mathbb{T}, d\xi/2\pi),$ where
$\mathbb{T}$ is the one-dimensional torus and $d\xi/2\pi$ is the
normalized measure.  One can verify the following formula
$$
\mathcal{F}(\mathcal{L}f)(\xi)=4\sin^{2}\frac{\xi}{2}\mathcal{F}(f)(\xi).
$$

The next result is obvious.
\begin{thm} The spectrum of the Laplace operator $\mathcal{L}$ on the one-dimensional
lattice $\mathbb{Z}$ is the interval  $[0, 4]$. A function $f$ belongs
to the space $PW_{\omega}(\mathbb{Z}), 0\leq \omega\leq 4,$ if and
only if the support of $\mathcal{F}f$ is a subset
 of $[-\pi,\pi)$ on which
$4\sin^{2}\frac{\xi}{2}\leq \omega$.

\end{thm}

 We consider  the cover $\Xi=\{\Omega_{j}\}$ of $\mathbb{Z}$ by disjoint sets $\Omega_{j}=\{\>j,\>j+1\}$ where $j$ runs over all even integers divisible by $3$: $\{..., -2, 0, 2, ... \}=2\mathbf{Z}$. 
 We treat every $\Omega_{j}$ as an induced graph whose set of vertices is $V(\Omega_{j})=\{\>j,\>j+1\},\>\>j\in 2\mathbf{Z},$ and which has only one  edge   $(j, \>j+1)$. Functional $\xi_{j}$ takes  form 
\begin{equation}\label{last-0}
 \langle f,\xi_{j}\rangle=\frac{1}{\sqrt{2}}\left(f(j)+f(j+1)\right),\>\>\>j\in 2\mathbf{Z,\>\>\>}f\in \ell^{2}(\mathbb{Z}).
\end{equation}
 One can check that  spectrum of the Laplace operator $\Delta_{j}$ on $\Omega_{j}$ defined by (\ref{L}) contains just  two values $\{0,\>2\}$. Thus $\Lambda_{\Xi}=2$.  For  an $0<\omega<4$ and $\alpha>0$  condition (\ref{gamma}) takes form
\begin{equation}\label{last}
 \gamma=\frac{1+\alpha}{\alpha}\frac{\omega}{2}<1.
\end{equation}
Note, that since for a large $\alpha$ the fraction $(1+\alpha)/\alpha$ is close to $1$  the condition implies that  one can have any $0<\omega<2$. 
As an application of Theorem \ref{conclusion} we obtain the following result.
\begin{thm} 
For every  $0<\omega<2$  every function $f\in PW_{\omega}(\mathcal{L})$ is uniquely determined by its average values (\ref{last-0}) and can be reconstructed from them in a stable way.

\end{thm} 
  In particular, if instead of infinite graph $\mathbb{Z}$ one would consider  a path graph $\mathbb{Z}_{N}$ whose eigenvalues are given by formulas $2-2\cos\frac{k\pi}{N},\>\>\>k=0,1, ..., N-1, $ the last Theorem  would mean that any eigenfunction with eigenvalue from a lower half of the spectrum  is uniquely determined and can be reconstructed from averages (\ref{last-0}).


\begin{thebibliography}{99}
\bibitem{A}
S. Albeverio, {\em Theory of Dirichlet forms and applications}, (English summary) Lectures on probability theory and statistics (Saint-Flour, 2000), 1106, Lecture Notes in Math., 1816, Springer, Berlin, 2003.

\bibitem{BH}
N. Bouleau, F. Hirsch, {\em Dirichlet forms and analysis on Wiener space}, Walter de Gruyter,
Berlin, New York, 1991.
\bibitem{CW1}
R. Coifman, G. Weiss, {\em Analyse harmonique non-commutative sur certains espaces homog`enes}, Lecture Notes in Math. 242, Springer-Verlag, 1971.

\bibitem{CW2}R. Coifman, G. Weiss, {\em Extensions of Hardy spaces and their use in analysis}, Bull. Amer. Math. Soc. 83 (1977), 569-645.


\bibitem{FOT}
M. Fukushima, Y. Oshima, M. Takeda, {\em Dirichlet forms and symmetric Markov processes}, De Gruyter Studies in Mathematics vol 19, Walter De Gruyter and co, Berlin 1994.

\bibitem{Pes00}
I. ~Pesenson, {\em A sampling theorem on homogeneous manifolds},
Trans. Amer. Math. Soc. {\bf 352} (2000), no. 9, 4257--4269.


\bibitem{Pes04a}
I.~Pesenson, {\em An approach to spectral problems on Riemannian
manifolds,}  Pacific J.  Math.  215/1 (2004), 183-199.


\bibitem{Pes04b}
I.~Pesenson,  {\em Poincare-type inequalities and reconstruction
of Paley-Wiener functions on manifolds, }  J. of Geometric Analysis, (4), 1, (2004), 101-121.


\bibitem {DS}
R. ~Duffin, A. Schaeffer, {\em A class of nonharmonic Fourier
series}, Trans. AMS, 72, (1952), 341-366.


\bibitem{FP}
 H. F\"uhr, Hartmut, I.  Pesenson, {\em Poincar\'e and Plancherel-Polya inequalities in harmonic analysis on weighted combinatorial graphs},  SIAM J. Discrete Math. 27 (2013), no. 4, 2007-2028.
 

\bibitem {Gr}
K.~Gr\"ochenig, {\em Foundations of Time-Frequency Analysis},
Birkh\"auser, 2001.

\bibitem{H}
S. Haeseler, M. Keller, D. Lenz, R. Wojciechowski, {\em Laplacians on infinite graphs: Dirichlet and Neumann boundary conditions},  J. Spectr. Theory 2
(2012), no. 4, 397-432.




\bibitem{Mo} B. Mohar, {\em Some applications of Laplace eigenvalues of graphs}, in G. Hahn and G. Sabidussi, editors, {\em Graph Symmetry: Algebraic Methods and Applications (Proc. Montr\'real 1996)}, volume 497 of {\rm Adv. Sci. Inst. Ser. C. Math. Phys. Sci.}, pp. 225-275, Dordrecht (1997), Kluwer.

\bibitem{Pes88}
I.~Pesenson, {\em The best approximation in a representation space
of a Lie group}, Dokl. Acad. Nauk USSR, v. 302/5 (1988),
1055-1059; (Engl. Transl. in Soviet Math. Dokl. 38/2
(1989), 384-388)
\bibitem{Pes98a}
I. Pesenson, {\em Sampling of Paley-Wiener functions on stratified
groups}, J. Four.\ Anal.\  Appl.   4 (1998), 269--280.


\bibitem{Pes01}
I.~Pesenson, {\em Sampling of band limited vectors}, J. Fourier
Anal. Appl. 7/1 (2001), 93-100.


\bibitem{Pes04b}
I.~Pesenson,  {\em Poincar\'{e}-type inequalities and reconstruction
of Paley-Wiener functions on manifolds, }  J.  Geometric Anal.
 4/1 (2004), 101-121.





 \bibitem{Pes08a}
 I. ~Pesenson, {\em  Sampling in Paley-Wiener spaces on combinatorial graphs}, Trans. Amer. Math. Soc. 360 (2008), no. 10, 5603-5627. 
 


\bibitem{Pes09d}
I. Z. Pesenson, {\em Variational splines and Paley-Wiener spaces on combinatorial graphs}, Constr. Approx. 29 (2009), no. 1, 1-21.

\bibitem{Pes10}
I. Z. Pesenson, M. Z. Pesenson, {\em Sampling, filtering and sparse approximations on combinatorial graphs}, J. Fourier Anal. Appl. 16 (2010), no. 6, 921-942.

\bibitem{S1}
K.T. Sturm, {\em Analysis on local Dirichlet spaces I. Recurrence, conservativeness and Lp-Liouville properties}, J. Reine Angew. Math. 456 (1994), 173-196.

\bibitem{S2} 
K.T. Sturm, {\em Analysis on local Dirichlet spaces II. Upper Gaussian estimates for the
fundamental solutions of parabolic equations}, Osaka J. Math. 32 (1995), 275-312.

\bibitem{S3} 
K.T. Sturm, {\em Analysis on local Dirichlet spaces III. The parabolic Harnack inequality}, J.
Math. Pures Appl. 75 (1998), 273-297.

\bibitem{W}Xiaohan Wang, Pengfei Liu, Yuantao Gu
{\em Local-Set-Based Graph Signal Reconstruction},
IEEE Transactions on Signal Processing (2015)

\end{thebibliography}
\end{document}